\documentclass[reqno]{amsart}

\usepackage{amsmath,amssymb,amsthm,mathrsfs}

\theoremstyle{plain}
\numberwithin{equation}{section}

\newcommand{\BC}{{\mathbb C}}

\newcommand{\BZ}{{\mathbb Z}}
\newcommand{\cB}{{\mathcal B}}

\newcommand{\cF}{{\mathcal F}}
\newcommand{\cG}{{\mathcal G}}\newcommand{\cH}{{\mathcal H}}

\newcommand{\cK}{{\mathcal K}}\newcommand{\cL}{{\mathcal L}}

\newcommand{\cS}{{\mathcal S}}\newcommand{\cT}{{\mathcal T}}
\newcommand{\cV}{{\mathcal V}}
\newcommand{\cW}{{\mathcal W}}\newcommand{\cX}{{\mathcal X}}
\newcommand{\cY}{{\mathcal Y}}\newcommand{\cZ}{{\mathcal Z}}

\newcommand{\fK}{{\mathfrak K}}

\newcommand{\wtilB}{\widetilde{B}}

\newcommand{\wtilS}{\widetilde{S}}

\newcommand{\wtilX}{\widetilde{X}}

\newcommand{\la}{\lambda}

\newcommand{\ran}{\textup{Ran\,}}

\newcommand{\im}{\textup{Im\,}}

\newcommand{\kr}{\textup{Ker\,}}
\newcommand{\cokr}{\textup{Coker\,}}

\newcommand{\ind}{\textup{Index}}

\newcommand{\mat}[2]{\ensuremath{\left[\begin{array}{#1}#2\end{array} \right]}}

\newcommand{\sbm}[1]{\left[\begin{smallmatrix} #1\end{smallmatrix}\right]}

\newcommand{\wtil}[1]{{\widetilde{#1}}}
\newcommand{\what}[1]{{\widehat{#1}}}

\newcommand{\ands}{\quad\mbox{and}\quad}

\theoremstyle{plain}
\newtheorem{theorem}{Theorem}[section]
\newtheorem{corollary}[theorem]{Corollary}
\newtheorem{lemma}[theorem]{Lemma}
\newtheorem{proposition}[theorem]{Proposition}
\theoremstyle{definition}
\newtheorem{definition}[theorem]{Definition}
\newtheorem{example}[theorem]{Example}
\newtheorem{remark}[theorem]{Remark}

\newtheorem{problem}[theorem]{Problem}

\begin{document}

\title{Equivalence after extension and Schur coupling for relatively regular operators}

\author[S. ter Horst]{S. ter Horst}
\address{S. ter Horst, School of Mathematical and Statistical Sciences,
North-West University,
Research Focus Area: Pure and Applied Analytics,
Private Bag X6001,
Potchefstroom 2520,
South Africa and DSI-NRF Centre of Excellence in Mathematical and Statistical Sciences (CoE-MaSS)}
\email{Sanne.TerHorst@nwu.ac.za}

\author[M. Messerschmidt]{M. Messerschmidt}
\address{M. Messerschmidt, Department of Mathematics and Applied Mathematics; University of Pretoria; Private bag X20 Hatfield; 0028 Pretoria; South Africa}
\email{miek.messerschmidt@up.ac.za}

\author[A.C.M. Ran]{A.C.M. Ran}
\address{A.C.M. Ran, Department of Mathematics, Faculty of Sciences, Vrije Universiteit Amsterdam, De Boelelaan 1111, 1081 HV Amsterdam, The Netherlands and Research Focus Area: Pure and Applied Analytics, North-West University, Potchefstroom, South Africa}
\email{a.c.m.ran@vu.nl}

\thanks{This work is based on the research supported in part by the National Research Foundation of South Africa (Grant Numbers 118513 and 127364).}



\begin{abstract}
It was recently shown in \cite{tHMRR19} that the Banach space operator relations Equivalence After Extension (EAE) and Schur Coupling (SC) do not coincide by characterizing these relations for operators acting on essentially incomparable Banach spaces. The examples that prove the non-coincidence are Fredholm operators, which is a subclass of relatively regular operators, the latter being operators with complementable kernels and ranges. In this paper we analyse the relations EAE and SC for the class of relatively regular operators, leading to an equivalent Banach space operator problem from which we derive new cases where EAE and SC coincide and provide a new example for which EAE and SC do not coincide and where the Banach space are not essentially incomparable.
\end{abstract}

\maketitle

\section{Introduction}\label{S:Intro}

Equivalence After Extension (EAE) and Schur Coupling (SC) are two relations on (bounded linear) Banach space operators that originated in the study of integral operators \cite{BGK84}, along with the relation Matricial Coupling (MC), and which have since found many other applications, cf., \cite{CP16,CP17,CDS14,ET17,S17,GKR17,tHMRRW18} for a few recent references. The applications of these  operator relations often rely on the fact that EAE, MC and SC coincide, in the specific context of the application, and it is important that one can easily and explicitly move between the three operator relations. This led to the question, posed by H. Bart and V.\'{E}. Tsekanovskii in \cite{BT94}, whether these three operator relations might coincide at the level of general Banach space operators. In fact, by then it was known that EAE and MC coincide \cite{BGK84,BT92a} and in \cite{BT94}, see also \cite{BT92b}, it was proved that EAE and MC are implied by SC. All these implications are obtained by explicit constructions, see Section 2 in \cite{tHR13} for an overview. Various attempts to prove that EAE (or MC) implies SC followed and several positive results in special cases were obtained \cite{BGKR05,tHR13,T14,tHMRRW18}. However, in the recent paper \cite{tHMRR19} an explicit counterexample showing that EAE need not imply SC was obtained from the characterization of EAE and SC on Banach spaces that are essentially incomparable. This reiterated the observation from \cite{tHMR15} that the Banach space geometries of the underlying spaces play an important role. The counterexample of \cite{tHMRR19} involves Fredholm operators, which motivated us into a further investigation of EAE and SC for this class of operators, and more generally for the class of relatively regular operators, without the essential incomparability assumption.

In order to describe the content of this paper in some more detail, we introduce some notation and terminology. Throughout, $U\in\cB(\cX)$ and $V\in\cB(\cY)$ will be two Banach space operators. Here, for Banach spaces $\cV$ and $\cW$, with $\cB(\cV,\cW)$ we indicate the Banach space of bounded linear operators mapping $\cV$ into $\cW$, abbreviated to $\cB(\cV)$ in case $\cV=\cW$. In the sequel, the term operator will always mean bounded linear operator, and invertibility of an operator will mean the operator has a bounded inverse.

We say that the operators $U$ and $V$ are {\em equivalent after extension (EAE)} if there exist Banach spaces $\cX_0$ and $\cY_0$ and invertible operators $E\in\cB(\cY\oplus\cY_0, \cX\oplus\cX_0)$ and $F\in\cB(\cX\oplus\cX_0,\cY\oplus\cY_0)$ so that
\begin{equation}\label{EAE}
\mat{cc}{U&0\\0&I_{\cX_0}}=E\mat{cc}{V&0\\0& I_{\cY_0}}F.
\end{equation}
In other words, $U\oplus I_{\cX_0}$ and $V\oplus I_{\cY_0}$ are equivalent. The operators $U$ and $V$ are called {\em Schur coupled (SC)} if there exists an operator matrix $M=\sbm{A&B\\C&D}\in\cB(\cX\oplus\cY,\cX\oplus\cY)$ so that $A$ and $D$ are invertible and
\[
U=A-BD^{-1}C \ands V=D-CA^{-1}B.
\]
Thus, the Schur complements of $M$ with respect to $D$ and $A$ are the operators $U$ and $V$, respectively.

In this paper we investigate the notions of EAE and SC, and in particular the question whether EAE of $U$ and $V$ implies SC of $U$ and $V$, for the class of {\em relatively regular operators}, that is, operators which have complemented kernels and ranges. In particular, their ranges must be closed. The terminology relatively regular goes back to Atkinson \cite{A53}, but this class of operators has also figured under the monikers generalized Fredholm operators \cite{C74} and generalized invertible operators. More generally, assume that $\kr U$ and $\kr V$ are complementable and so are the closures of $\ran U$ and $\ran V$. Equivalently, $\cX$ and $\cY$ admit decompositions $\cX_1\oplus\cX_2=\cX=\cX_1'\oplus \cX_2'$ and $\cY_1\oplus\cY_2=\cY=\cY_1'\oplus \cY_2'$ such that with respect to these decompositions $U$ and $V$ take the form
\begin{align}
U=\mat{cc}{U'&0\\ 0&0}:\mat{c}{\cX_1\\ \cX_2}\to \mat{c}{\cX_1'\\ \cX_2'} &\ands
V=\mat{cc}{V'&0\\ 0&0}:\mat{c}{\cY_1\\ \cY_2}\to \mat{c}{\cY_1'\\ \cY_2'},\notag\\
\mbox{with}\quad\cX_2=\kr U,\quad \cX_1'=\overline{\ran U},&\quad
\cY_2=\kr V,\quad \cY_1'=\overline{\ran V}.\label{UVdec}
\end{align}
Operators $U$ and $V$ that admit a decomposition as above will be called {\em complementable}. Thus  $U$ and $V$ are relatively regular if they are complementable and have closed range. A {\em Fredholm operator} is a relatively regular operator with finite dimensional kernel and cokernel, while a {\em left (respectively, right) Atkinson} operator is a relatively regular operator with finite dimensional kernel (respectively, cokernel). Here, with some abuse of terminology we refer to a complement of the range of a relatively regular operator $T$ as its {\em cokernel}, denoted as $\cokr T$. Note that all left (respectively, right) Atkinson operators are left (respectively, right) Fredholm operators, but that the converse need not hold, since left and right Fredholm operators need not be relatively regular; see \cite[Chapter 7]{A04} for more details.

In Section \ref{S:EAE-SR}, for given complementable operators $U$ and $V$ that are EAE we characterise the invertible operators $E$ and $F$, of a special form identified in \cite{tHR13}, cf., \eqref{EFform1} below, that establish the EAE of $U$ and $V$. Specifying to the case of relatively regular operators, we recover the result from \cite{BT92a} that $U$ and $V$ are EAE if and only if the kernels of $U$ and $V$ are isomorphic and the cokernels of $U$ and $V$ are isomorphic. However, more importantly, it enables us to identify a Banach space operator problem which is in some way equivalent to the question whether the EAE operators $U$ and $V$ are also SC. We view this observation as the main contribution of our paper. The Banach space operator problem is presented in Section \ref{S:EquivBSOP}, as Problem \ref{P:BSOP}, and in that section we also explain the connection with the question whether given relatively regular operators $U$ and $V$ that are EAE are also SC.

Some analysis of Problem \ref{P:BSOP} is conducted in Section \ref{S:AnalyseBSOP}, and the implications of the obtained results are translated into the context of EAE and SC in Section \ref{S:EAEvsSC}. Again, Banach space geometry plays an important role. In particular, in Section \ref{S:BSP} we further analyse a Banach space property (which we already encountered in \cite{tHMRR19}) that plays an important role in the analysis of Section \ref{S:AnalyseBSOP}, and which turns out to be equivalent to the existence of Fredholm operators of a certain index. Although this analysis in Section \ref{S:AnalyseBSOP} does not provide a complete answer to Problem \ref{P:BSOP}, it shows some of the intricacies that occur when considering the question whether EAE implies SC beyond the known cases, e.g., Hilbert spaces or essentially incomparable Banach spaces, even when restricting to relatively regular operators.

We conclude this introduction with some comments on notation and terminology and recall some facts from Banach space (operator) theory. Throughout this paragraph let $\cG$ and $\cH$ be Banach spaces. An operator $T\in\cB(\cG,\cH)$ is called {\em inessential} whenever $I_\cG-ST$ is Fredholm for any $S\in\cB(\cH,\cG)$, or, equivalently, $I_\cH-TS$ is Fredholm for any $S\in\cB(\cH,\cG)$; in fact, in this case the operators $I_\cG-ST$ and $I_\cH-TS$ are Fredholm with index 0. The class of inessential operators contains all compact, strictly singular and strictly co-singular operators and is closed under left and right multiplication, that is, if $T\in\cB(\cG,\cH)$ is inessential, then so is $STR$ for all $S\in\cB(\cH,\cL)$ and $R\in\cB(\cK,\cG)$, with $\cL$ and $\cK$ arbitrary Banach spaces. See Sections 6.1 and 6.2 of \cite{A04} for further details.
The Banach spaces $\cG$ and $\cH$ are called {\em essentially incomparable} in case all operators in $\cB(\cG,\cH)$ are inessential, or, equivalently, all operators in $\cB(\cH,\cG)$ are inessential. In case $\cG$ and $\cH$ are essentially incomparable, they are also {\em projection incomparable} which means that there is no infinite dimensional complemented subspace of $\cG$ that is isomorphic to a complemented subspace of $\cH$; see Section 7.5 in \cite{A04} for further discussion and many examples.

\section{Equivalence after extension for relatively regular operators}\label{S:EAE-SR}

Recall from \cite{tHR13}, see also \cite[Lemma 2.1]{tHMRRW18}, that if $U\in\cB(\cX)$ and $V\in\cB(\cY)$ are EAE, then it is always possible to select the operators $E$ and $F$ in \eqref{EAE} with $\cX_0=\cY$, $\cY_0=\cX$ and so that $E$ and $F$ have the form
\begin{equation}\label{EFform1}
\begin{aligned}
F=\mat{cc}{F_{11}& I_\cY\\ I+F_{22}F_{11} & F_{22}},&\quad
E=\mat{cc}{E_{11} & U\\ E_{21} & -F_{11}}\\
F^{-1}=\mat{cc}{-F_{22}& I_\cX\\ I+F_{11}F_{22} & -F_{11}},&\quad
E^{-1}=\mat{cc}{\what{E}_{11} & V\\ \what{E}_{21} & F_{22}}.
\end{aligned}
\end{equation}
In Theorem \ref{T:EAEcomplement} below, for complementable operators $U$ and $V$ that are EAE we completely characterize the invertible operators $E$ and $F$ of the form \eqref{EFform1} that establish the EAE of $U$ and $V$, and in Theorem \ref{T:EAErelativelyreg} below we specialize this result to the case where $U$ and $V$ are relatively regular.

By Proposition 1 in \cite{BT92a}, a necessary condition for complementable operators $U$ and $V$ as in \eqref{UVdec} to be EAE is that $\cX_2$ and $\cY_2$ are isomorphic and $\cX_2'$ and $\cY_2'$ are isomorphic, that is, there exist operators
\begin{equation}\label{E'F'}
E'\in\cB(\cX_2,\cY_2)
\ands F'\in\cB(\cX_2',\cY_2'),\quad\mbox{both invertible}.
\end{equation}
While for complementable operators this condition need not be sufficient, by Example 6 in \cite{BT92a}, it is sufficient if one restricts to relatively regular operators \cite[Theorem 2]{BT92a}. We recover this result in Theorem \ref{T:EAErelativelyreg} below.

Assuming we have invertible operators $E'$ and $F'$ as in \eqref{E'F'}, we consider operators $E$ and $F$ as in \eqref{EFform1} that are of the following form
\begin{equation}\label{EFform2}
\begin{aligned}
F&=\mat{cc|cc}{
Y_1 & 0 & I_{\cY_1} & 0 \\ Y_2 & E' & 0 & I_{\cY_2} \\\hline
I+Y_3Y_1 & 0 & Y_3 & 0 \\ Y_4Y_1-E'^{-1}Y_2 & 0 & Y_4 & -E'^{-1}}
: \mat{c}{\cX_1\\\cX_2\\
\hline \cY_1\\\cY_2}\to \mat{c}{\cY_1\\\cY_2\\\hline \cX_1\\\cX_2},\\
E&=\mat{cc|cc}{ X_1 & X_2 & U' & 0 \\ 0 & F'^{-1} & 0 & 0 \\\hline
X_5 
 &X_3 & -Y_1 & 0 \\
X_6 
& X_4 & -Y_2 & -E' }
: \mat{c}{\cY_1'\\\cY_2'\\ \hline \cX_1\\\cX_2}\to \mat{c}{\cX_1'\\\cX_2'\\ \hline \cY_1\\\cY_2}.
\end{aligned}
\end{equation}
One easily verifies that an operator $F$ of this form is invertible, with inverse given by
\begin{equation}\label{Finvform}
F^{-1}=\mat{cc|cc}{
-Y_3 & 0 & I_{\cX_1} & 0 \\ -Y_4  & E'^{-1} & 0 & I_{\cX_2} \\ \hline
I_{\cY_1}+Y_1Y_3  & 0 & -Y_1 & 0 \\ Y_2Y_3+E'Y_4 & 0 & -Y_2 & -E'}
: \mat{c}{\cY_1\\\cY_2\\ \hline\cX_1\\\cX_2} \to \mat{c}{\cX_1\\\cX_2\\ \hline\cY_1\\\cY_2},
\end{equation}
while the invertibility of $F'$ and $E'$ imply that $E$ is invertible if and only if the block operator
\begin{equation}\label{Eblock}
\mat{cc}{X_1 & U'\\ X_5 & -Y_1}
\end{equation}
is invertible.
The motivation of considering $E$ and $F$ in this special form becomes clear from the following result.

\begin{theorem}\label{T:EAEcomplement}
Let $U\in\cB(\cX)$ and $V\in\cB(\cY)$ be complementable Banach space operators as in \eqref{UVdec}. Assume $U$ and $V$ are EAE with the operators $E$ and $F$ that establish the EAE of the form \eqref{EFform1}. Then $E$ and $F$ are of the form \eqref{EFform2} with $E'$ and $F'$ invertible operators as in \eqref{E'F'} and such that
\begin{equation}\label{EAEcondition1}
\begin{aligned}
&\mat{cc}{X_1 & U'\\ X_5 & -Y_1}\mat{cc}{V' & 0 \\ 0 & I_{\cX_1}}= \mat{cc}{U' & 0 \\ 0 & I_{\cY_1}}\mat{cc}{-Y_3 & I_{\cX_1} \\ I_{\cY_1} +Y_1 Y_3 & - Y_1}\\[2mm]
&\qquad \qquad \mbox{and}\quad Y_4=F'^{-1}(X_6V' -Y_2Y_3).
\end{aligned}
\end{equation}

Conversely, if there exist invertible operators $E'$ and $F'$ as in \eqref{E'F'} and operators $X_1, X_5, Y_1, Y_3$ such that the block operator \eqref{Eblock} is invertible and satisfies the first identity in \eqref{EAEcondition1}, then $U$ and $V$ are EAE and the EAE of $U$ and $V$ is established by $E$ and $F$ given by \eqref{EFform2}, where the operators $X_2,X_3,X_4, X_6,Y_1,Y_2,Y_3$ can be chosen arbitrarily and $Y_4$ is given by the second identity in \eqref{EAEcondition1}.
\end{theorem}

\begin{remark}
The relation between $U'$ and $V'$ in the first line of \eqref{EAEcondition1} and with the operator \eqref{Eblock} invertible means that $U'\in\cB(\cX_1,\cX_1')$ and $V'\in\cB(\cY_1,\cY_1')$ are EAE in the `non-square' sense of \cite{BGK84}, i.e., with possibly $\cX_1\neq\cX_1'$ and $\cY_1\neq\cY_1'$. The relations EAE and MC were originally studied for Banach space operators $U\in\cB(\cX,\cX')$ and $Y\in\cB(\cY,\cY')$, cf., \cite{BGK84}, however, for SC to make sense it is required that $\cX\simeq\cX'$ and $\cY\simeq\cY'$, so that one usually assumes $\cX=\cX'$ and $\cY=\cY'$.
\end{remark}

%

\begin{proof}[\bf Proof of Theorem \ref{T:EAEcomplement}]
Part of the first claim, under the assumption that $U$ and $V$ are relatively regular, was already given in the proof of \cite[Proposition 3.2]{tHR13}, but we will repeat the arguments here for completeness. Assume $U$ and $V$ as in \eqref{UVdec} are EAE with $E$ and $F$ of the form \eqref{EFform1}. Note that it is clear that the right upper corner of $F$ is as in \eqref{EFform2}, and that once it is established that $F_{11}$ and $F_{22}$ are as in \eqref{EFform2}, then the left lower corner of $F$ in \eqref{EFform2} as well as the formula for ${F^{-1}}$ in \eqref{Finvform} also follow. Furthermore, it then also follows that the right upper and right lower corner of $E$ are as in \eqref{EFform2}.

The fact that $E$ and $F$ are invertible operators satisfying \eqref{EAE} implies that $F$ maps $\cX_2=\kr U$ onto $\cY_2=\kr V$ and $E^{-1}$ maps any complement of $\im V$ onto any complement of $\im U$, in particular, $E^{-1}$ maps $\cY_2'$ onto $\cX_2'$. This establishes the zero entries in the second block column of $F$ and in the second block row of $E$ in \eqref{EFform2}, as well as the fact that the remaining entries, $F'^{-1}\in\cB(\cY_2,\cX_2)$ and $E'\in\cB(\cX_2',\cY_2')$ are invertible. Write $F_{22}=\sbm{Y_3 & Y_5\\ Y_4 & Y_6}$, compatible with the decomposition in \eqref{EFform2}. The fact that the left lower corner in $F$ is given by $I+F_{22}F_{11}$ and has zero entries in the right upper and right lower corner yields
\[
Y_5F'=0 \ands I+ Y_6F'=0.
\]
Since $F'$ is invertible, we then obtain that $Y_5=0$ and $Y_6=-F'^{-1}$. Hence $F$ is of the form \eqref{EFform2}. In passing, we also showed that $E$ is as in \eqref{EFform2}. The fact that $-E'$ and $F'^{-1}$ are invertible and appear in $E$ in a block column and block row, respectively, that otherwise only contains zero-operators implies that $E$ is invertible if and only if the $2 \times 2$ block operator matrix obtained by removing the rows and columns that contain $-E'$ and $F'^{-1}$ is invertible. To see this, note that this operator matrix appears after taking the Schur complement of $E$ with respect to $-E'$ and then the Schur complement with respect to $F'^{-1}$. The resulting $2 \times 2$ block operator matrix is as in \eqref{Eblock}, so that we obtain that $E$ is invertible if and only if the operator in \eqref{Eblock} is invertible. Now, writing out the EAE relation \eqref{EAE} with $E$ and $F$ as in \eqref{EFform2}, it follows directly that all that remains to be verified is \eqref{EAEcondition1}. Thus we have proved the first claim of Theorem \ref{T:EAEcomplement}.

For the converse direction, assume that $X_1, X_5, Y_1, Y_3$ are operators such that the block operator \eqref{Eblock} is invertible, and the first identity in \eqref{EAEcondition1} holds. One then easily verifies that with $E$ and $F$ as in \eqref{EFform2}, with $X_2,X_3,X_4, X_6,Y_1,Y_2,Y_3$ arbitrary and $Y_4$ given by \eqref{EAEcondition1}, the EAE relation \eqref{EAE} holds.
\end{proof}

When specialized to the case that $U$ and $V$ are relatively regular, i.e., $U'$ and $V'$ invertible, we obtain the following result.
The necessary and sufficient conditions for EAE of two relatively regular operators were also given in Theorem 2 in \cite{BT92a}, however, only with a sketch of the proof. The characterization of the operators $E$ and $F$ that establish the EAE relation appears to be new.

\begin{theorem}\label{T:EAErelativelyreg}
Let $U\in\cB(\cX)$ and $V\in\cB(\cY)$ be relatively regular Banach space operators as in \eqref{UVdec}, i.e., with $U'$ and $V'$ invertible. Then $U$ and $V$ are EAE if and only if there exist invertible operators $E'$ and $F'$ as in \eqref{E'F'}. Moreover, in that case, the operators $E$ and $F$ of the form \eqref{EFform1} that establish the EAE of $U$ and $V$ are precisely the operators of the form \eqref{EFform2} where $E'$ and $F'$ are arbitrary invertible operators as in \eqref{E'F'}, the operators $Y_1,Y_2,Y_3,X_2,X_3,X_4, X_6$ can be chosen arbitrarily, and the operators $X_1$ $X_5$ and $Y_4$ are given by
\begin{equation}\label{X1X5Y4}
X_1= -U' Y_3 V'^{-1},\quad
X_5= (I + Y_1Y_3)V'^{-1},\quad
Y_4=E'^{-1}(X_6V' -Y_2Y_3).
\end{equation}
\end{theorem}

\begin{proof}[\bf Proof]
According to Proposition 1 in \cite{BT92a}, when $U$ and $V$ are EAE, invertible operators $E'$ and $F'$ as in \eqref{E'F'} exist. For the converse direction, assume we have invertible operators $E'$ and $F'$ as in \eqref{E'F'}. Using the converse implication in Theorem \ref{T:EAEcomplement}, it suffices to find operators $X_1, X_5, Y_1, Y_3$ such that the block operator \eqref{Eblock} is invertible and the first identity in \eqref{EAEcondition1} holds. Since $U'$ and $V'$ are invertible, we can choose $Y_1$ and $Y_3$ freely, and define $X_1$ and $X_5$ via the identity
\begin{equation}\label{EblockFact}
\mat{cc}{X_1 & U'\\ X_5 & -Y_1}= \mat{cc}{U' & 0 \\ 0 & I}\mat{cc}{-Y_3 & I \\ I +Y_1 Y_3 & - Y_1}\mat{cc}{V'^{-1} & 0 \\ 0 & I},
\end{equation}
which implies that first identity in \eqref{EAEcondition1} holds, and also yields the formulas in \eqref{X1X5Y4}. Since $E'$ and $F'$ are invertible as well as the middle factor in the factorization \eqref{EblockFact} of the block operator \eqref{Eblock}, it follows that the block operator \eqref{Eblock} is invertible. Hence, by Theorem \ref{T:EAEcomplement} we obtain that $U$ and $V$ are EAE. Furthermore, the second part of Theorem \ref{T:EAEcomplement} also tells us that the operators $E$ and $F$ of the form \eqref{EFform1} are as claimed, since the relation \eqref{EAEcondition1} is equivalent to the identities \eqref{X1X5Y4} when $U'$ and $V'$ are invertible.
\end{proof}

Recall that $U$ and $V$ are called {\em strongly equivalent after extension (SEAE)} in case $U$ and $V$ are EAE in such a way that the operators $E$ and $F$ that establish the EAE relation have the left lower corner (from $\cX_0$ to $\cY$) and right upper corner (from $\cY$ to $\cX_0$) invertible, respectively. By Theorem 2 in \cite{BT92b}, see also Theorem 2.4 in \cite{tHR13}, $U$ and $V$ are SEAE if and only if they are SC. It is not directly clear that if $U$ and $V$ are SEAE, then there are also $E$ and $F$ that establish the SEAE and are of the form \eqref{EFform1}. We next show that this is the case.

\begin{lemma}\label{L:SEAEspecialform}
Let $U\in\cB(\cX)$ and $V\in\cB(\cY)$ be SEAE. Then the SEAE can be established with $E$ and $F$ of the form \eqref{EFform1}.
\end{lemma}

\begin{proof}[\bf Proof]
Let $E$ and $F$ be invertible operators that establish the SEAE of $U$ and $V$. Decompose $E$ and $F$ as block operator matrices $E=[E_{ij}]_{i,j=1,2}$ and $F=[F_{ij}]_{i,j=1,2}$ compatible with the EAE identity \eqref{EAE}. Then $F_{12}$ and $E_{21}$ are invertible. In the transfer to the special form \eqref{EFform1}, by Lemma 4.1 in \cite{tHR13} the new $(1,2)$-entry of $F$ becomes $I_\cY$ and the new $(2,1)$-entry of $E$ becomes $F_{12}E_{21}$, both of which are invertible. It then follows as in the proofs of Corollary 4.2 in \cite{tHR13} and Lemma 2.1 in \cite{tHMRRW18} that the new operators that establish the SEAE of $U$ and $V$ are also of the form \eqref{EFform1}.
\end{proof}

\section{An equivalent Banach space operator problem}\label{S:EquivBSOP}

In this section we introduce the following Banach space operator problem which is related to the question whether EAE and SC coincide for relatively regular operators, in a way described below.

\begin{problem}\label{P:BSOP}
Consider Banach spaces $\cV$, $\cW$, $\cZ_1$ and $\cZ_2$ so that
\begin{equation}\label{BanProps}
\cV \oplus \cZ_1 \simeq \cV \oplus \cZ_2,\quad  \cW \oplus \cZ_1 \simeq \cW \oplus \cZ_2.
\end{equation}
Find operators
\begin{equation}\label{Ops}
\begin{aligned}
&A_{12}\in\cB(\cZ_1,\cV),\quad
A_{21}\in\cB(\cV,\cZ_2),\quad
A_{22}\in\cB(\cZ_1,\cZ_2),\\
&\qquad\qquad B_1\in\cB(\cW,\cV),\quad B_2\in\cB(\cV,\cW),
\end{aligned}
\end{equation}
so that the $2 \times 2$ block operator
\begin{equation}\label{T}
T=\mat{cc}{I_\cV-B_1B_2 & A_{12} \\ A_{21} & A_{22}}:\mat{c}{\cV\\ \cZ_1}\to\mat{c}{\cV\\ \cZ_2}
\end{equation}
is invertible.
\end{problem}

If such an invertible operator $T$ exists, then we say that Problem \ref{P:BSOP} associated with the spaces $\cV$, $\cW$, $\cZ_1$, $\cZ_2$ is solvable, or simply that Problem \ref{P:BSOP} is solvable in case there should not be any unclarity about the spaces.

Note that by assumption \eqref{BanProps} invertible operators between $\cV\oplus\cZ_1$ and $\cV\oplus\cZ_2$ exist, the question is whether there exists one of the form \eqref{T}.

We will now explain the connection with EAE and SC. Assume $U$ and $V$ are relatively regular operators as in \eqref{UVdec}. Assume $U$ and $V$ are EAE. We want to know when they are also SC, or, equivalently, SEAE. According to Theorem \ref{T:EAErelativelyreg}, all operators of the form \eqref{EFform1} that establish the EAE relation of $U$ and $V$ are given by \eqref{EFform2} with $Y_1$, $Y_2$, $Y_3$, $X_2$, $X_3$, $X_4$, $X_6$ arbitrary operators (acting between the spaces indicated in \eqref{EFform2}) and $X_1$, $X_5$ and $Y_4$ are given by \eqref{X1X5Y4}. By Lemma \ref{L:SEAEspecialform}, the question whether $U$ and $V$ are also SEAE now reduces to the question whether we can find $E$ and $F$ of the form \eqref{EFform1} with $E_{21}$ invertible. Now set
\[
A_{12}=X_3,\quad A_{21}=X_6 V',\quad A_{22}=X_4,\quad B_1=Y_1,\quad B_2= Y_3,
\]
so that
\begin{equation}\label{VWZ1Z2def}
\cV=\cY_1,\quad \cW=\cX_1,\quad \cZ_1=\cY_2, \quad \cZ_2=\cY_2'.
\end{equation}
It is then clear from Theorem \ref{T:EAErelativelyreg} that $E_{21}$ and $T$ in \eqref{T} are related through
\[
E_{21}\sbm{V'& 0 \\ 0& I}= T,
\]
and thus $E_{21}$ is invertible if and only if $T$ is invertible. Since in the choices for $E$ and $F$ we are free to choose $X_3$, $X_4$, $X_6$, $Y_1$ and $Y_3$, to see that the problem whether $U$ and $V$ are SEAE reduces to the above Banach space operator problem, it remains to show that for our choice of the spaces $\cV$, $\cW$, $\cZ_1$ and $\cZ_2$ condition \eqref{BanProps} is satisfied. Note that $U$ and $V$ being relatively regular, i.e., $U'$ and $V'$ invertible, yields $\cX_1 \simeq \cX_1'$ and $\cY_1 \simeq \cY_1'$, while the assumption that $U$ and $V$ are EAE gives $\cX_2\simeq \cY_2$ and $\cX_2'\simeq \cY_2'$. This implies that
\[
\mat{c}{\cV\\\cZ_1}=\mat{c}{\cY_1\\\cY_2}=\cY=\mat{c}{\cY_1'\\\cY_2'}\simeq \mat{c}{\cY_1\\\cY_2'}=\mat{c}{\cV\\ \cZ_2}
\]
and
\[
\mat{c}{\cW\\\cZ_1}=\mat{c}{\cX_1\\ \cY_2}\simeq\mat{c}{\cX_1\\ \cX_2}=\cX=\mat{c}{\cX_1'\\\cX_2'}\simeq \mat{c}{\cX_1\\ \cY_2'}=\mat{c}{\cW\\ \cZ_2}.
\]
Hence condition \eqref{BanProps} is indeed satisfied.

We summarise the above discussion in the following proposition.

\begin{proposition}
Let $U\in\cB(\cX)$ and $V\in\cB(\cY)$ be relatively regular operator as in \eqref{UVdec} that are EAE.  Then $\cV$, $\cW$, $\cZ_1$ and $\cZ_2$ defined in \eqref{VWZ1Z2def} satisfy \eqref{BanProps}. Moreover, $U$ and $V$ are SEAE if and only if there exists an invertible operator matrix $T$ as in \eqref{T} with operator entries as in \eqref{Ops}.
\end{proposition}

The above shows how for any relatively regular EAE operators $U$ and $V$ a version of Problem \eqref{P:BSOP} can be set up in such a way that solvability of this problem coincides with $U$ and $V$ being SEAE. It is also the case that with any version of Problem \ref{P:BSOP} we can associate relatively regular EAE operators $U$ and $V$ so that Problem \ref{P:BSOP} is the derived problem for these operators $U$ and $V$. Specifically, for Banach spaces $\cV$, $\cW$, $\cZ_1$ and $\cZ_2$ satisfying \eqref{BanProps}, let $U'\in\cB(\cW)$ and $V'\in\cB(\cV)$ be invertible operators (in particular $U'=I_\cW$ and $V'=I_\cV$ is a possibility), set $\cX=\cW\oplus\cZ_1$ and $\cY=\cV\oplus\cZ_1$, and define
\begin{equation}\label{UCcounter}
\begin{aligned}
U&=\mat{cc}{U'&0\\0&0}:\mat{c}{\cW\\\cZ_1}\to \mat{c}{\cW\\\cZ_2}\simeq \mat{c}{\cW\\\cZ_1},\\
V&=\mat{cc}{V'&0\\0&0}:\mat{c}{\cV\\\cZ_1}\to \mat{c}{\cV\\\cZ_2}\simeq \mat{c}{\cV\\\cZ_1}.
\end{aligned}
\end{equation}
Clearly $U$ and $V$ are EAE, by Theorem \ref{T:EAErelativelyreg}. If there exist operators $A_{12}$, $A_{21}$, $A_{22}$, $B_1$ and $B_2$ as in \eqref{Ops} such that $T$ in \eqref{T} is invertible, then $U$ and $V$ are SEAE, independently of the choice of $U'$ and $V'$, and if there are no such operators that make $T$ invertible, then $U$ and $V$ are not SEAE, again independently of the choice of $U'$ and $V'$. The value of this observation is that it provides a way to construct operators that are EAE but not SC.

\begin{corollary}\label{C:EAEnotSC}
Let $\cV$, $\cW$, $\cZ_1$ and $\cZ_2$ be Banach spaces satisfying \eqref{BanProps} so that the associated Problem \ref{P:BSOP} is not solvable. Then for any invertible operators $U'\in\cB(\cW)$ and $V'\in\cB(\cV)$, the operators $U$ and $V$ in \eqref{UCcounter} are EAE but not SEAE.
\end{corollary}

\section{A Banach space property}\label{S:BSP}

In this section we further analyse a Banach space property considered in \cite{tHMRR19}, which we will also encounter in the following sections.

\begin{definition}
Given Banach space $\cZ$ and an integer $k>0$, we say that $\cZ$ is {\em stable under finite dimensional quotients of dimension $k$}, abbreviated to $\cZ$ is {\em SUFDQ$_k$}, if for any finite dimensional subspace $\cF\subset \cZ$ with $\dim (\cF) =k$ we have $\cZ\simeq \cZ/\cF$. Furthermore, we say that $\cZ$ is {\em stable under finite dimensional sums of dimension $k$}, abbreviated to $\cZ$ is {\em SUFDS$_k$}, if $\cZ\simeq \cZ\oplus \cF$ for any finite dimensional Banach space $\cF$ with $\dim (\cF)=k$.
\end{definition}

Note that if $\cZ$ is SUFDQ$_k$ (respectively, SUFDS$_k$) then $\cZ$ is also SUFDQ$_{nk}$ (respectively SUFDS$_{nk}$) for all integers $n>0$. In particular, if $\cZ$ is SUFDQ$_1$ (respectively, SUFDS$_1$), then $\cZ$ is SUFDQ$_k$ (respectively, SUFDS$_k$) for all integers $k>0$. Therefore, we shall abbreviate SUFDQ$_1$ and SUFDS$_1$ by SUFDQ and SUFDS, respectively.

The Banach space properties SUFDQ and SUFDS appeared in \cite{tHMRR19}, and it was shown there that SUFDQ implies SUFDS \cite[Lemma 3.3]{tHMRR19}, as well as that the Banach space sum of two spaces that are SUFDQ is also SUFDQ \cite[Proposition 3.2]{tHMRR19}. As also observed in \cite{tHMRR19}, all primary Banach spaces are SUFDQ, hence they are SUFDQ$_k$ for all integers $k>0$.

\begin{proposition}\label{P:BanPropEquiv}
For all integer $k>0$ and Banach spaces $\cZ$ the following are equivalent:
\begin{itemize}
  \item[(i)] $\cZ$ is SUFDQ$_k$;
  \item[(ii)] $\cZ$ is SUFDS$_k$;
  \item[(iii)] in $\cB(\cZ)$ there exist Fredholm operators of index $k$;
  \item[(iv)] in $\cB(\cZ)$ there exist Fredholm operators of index $-k$.
\end{itemize}
\end{proposition}

\begin{proof}[\bf Proof]
The implication (i) $\Rightarrow$ (ii) follows by the same argument as in the proof of Lemma 3.3 in \cite{tHMRR19}, if one restricts the finite dimensional Banach spaces to those having dimension $k$.

Next we prove (ii) $\Rightarrow$ (i). Let $\cF\subset \cZ$ with $\dim (\cF)=k$ and let $\cW$ be a complement of $\cF$ in $\cZ$, so that $\cW\simeq \cZ/\cF$. Since $\cZ$ is SUFDS$_k$, we have
\[
\cW\oplus \cF=\cZ\simeq \cZ\oplus\cF.
\]
By Lemma 3.4 in \cite{tHMRR19}, which applies since $\dim (\cF)<\infty$, we have $\cZ/\cF\simeq \cW\simeq \cZ$. Thus (i) holds.

For the equivalence of (ii) and (iii), first assume (iii) holds and let $\cF$ be a Banach space of dimension $k$. Let $F\in\cB(\cZ)$ be a Fredholm operator with index $k$. Without loss of generality $F$ is surjective, so that $\dim (\kr F)=k$. Now let $\Xi\in\cB(\cZ,\cF)$ be any operator that maps $\kr F$ onto $\cF$ and is zero on a complement of $\kr F$ in $\cZ$. Then $\sbm{F\\\Xi}$ is an invertible operator in $\cB(\cZ,\cZ\oplus\cF)$, hence $\cZ\simeq\cZ\oplus \cF$. Thus (iii) implies (ii).

Conversely, assume (ii) holds. Let $\cF$ be any Banach space of dimension $k$. By assumption $\cZ\simeq\cZ\oplus\cF$. Now let $T$ be an isomorphism between $\cZ$ and $\cZ\oplus\cF$ and decompose $T$ as $T=\sbm{F\\\Xi}$ with $F\in \cB(\cZ,\cZ)$ and $\Xi\in \cB(\cZ,\cF)$. In particular, $T$ is Fredholm with index 0. Since $\Xi$ is a finite rank operator, $T_0:=\sbm{F\\ 0}=T-\sbm{0\\\Xi}$ is also Fredholm with index 0. Since $\dim (\kr T_0)=\dim (\kr F)$ and $\dim (\cokr T_0) =\dim (\cokr F) +k$, it follows that $F$ is Fredholm with index $k$.

Similarly one proves that (ii) and (iv) are equivalent. For the implication (ii) $\Rightarrow$ (iv), to see that there exists a Fredholm operator in $\cB(\cZ)$ with index $-k$ a similar argument applies based on an invertible operator in $\cB(\cZ\oplus\cF,\cZ)$. Conversely, assuming there exists a Fredholm operator with index $-k$ in $\cB(\cZ)$, one can modify the argument in the proof of the implication  (iii) $\Rightarrow$ (ii) to construct an invertible operator from $\cZ\oplus \cF$ to $\cZ$.
\end{proof}

We conclude this section with a few additional observations and comments.

\begin{lemma}\label{L:BSPsum}
Let $\cZ$ be SUFDQ$_k$. Then $\cW\oplus \cZ$ is SUFDQ$_k$ for any Banach space $\cW$.
\end{lemma}

\begin{proof}[\bf Proof]
This result is straightforward from the fact that $\cZ$ has property (ii) (or (iii)) in Proposition \ref{P:BanPropEquiv}.
\end{proof}

\begin{lemma}\label{L:BSPfindimcomp}
Let $\cZ$ be SUFDQ$_k$ and let $\cF\subset\cZ$ be finite dimensional. Then $\cZ/\cF$ is SUFDQ$_k$.
\end{lemma}

\begin{proof}[\bf Proof]
Let $\cZ_1$ be a complement of $\cF$ in $\cZ$. It suffices to prove that $\cZ_1$ is SUFDQ$_k$. Let $T\in\cB(\cZ)$ be a Fredholm operator of index $k$ and write $T=\sbm{T_{11}&T_{12}\\ T_{21}& T_{22}}$ with respect to the decomposition $\cZ=\cZ_1\oplus\cF$. Then $T_1:=\sbm{T_{11}&0\\0&0}=T-\sbm{0&T_{12}\\ T_{21}& T_{22}}$ is Fredholm with index $k$ because $\sbm{0&T_{12}\\ T_{21}& T_{22}}$ is finite rank. We have $\kr T_1 = \kr T_{11} \oplus \cF$ and $\cokr T_1 = \cokr T_{11} \oplus \cF$. Thus $\ind(T_{11})=\ind(T_{1})=k$ and $T_{11}\in\cB(\cZ_1)$. Hence $\cZ_1$ is SUFDQ$_k$.
\end{proof}

\begin{corollary}\label{C:BSPkTo-k}
For any integer $k$ and Banach space $\cZ$, $\cB(\cZ)$ contains Fredholm operators of index $k$ if and only if $\cB(\cZ)$ contains Fredholm operators of index $-k$.
\end{corollary}

\begin{corollary}\label{C:BSPlargerk}
Let $\cZ$ be a Banach space that has a complemented subspace which is SUFDQ$_k$. Then $\cB(\cZ)$ contains Fredholm operators of index $nk$ for all $n\in\BZ$. In particular, if $\cZ$ contains a complemented copy of a primary Banach space, then $\cB(\cZ)$ contains Fredholm operators of all indices.
\end{corollary}

\begin{remark}\label{R:BSexamples}
In \cite{GM97}, W.T. Gowers and B. Maurey constructed a Banach space $\cZ$ which is isomorphic to its subspaces of even codimension, but not to those of odd codimension. In other words, $\cZ$ is SUFDQ$_2$ but not SUFDQ, so that on $\cZ$ Fredholm operators exist, but only of even index.

Furthermore, there also exist Banach spaces where all Fredholm operators have index 0, i.e., so that there is no $k>0$ so that the Banach spaces is SUFDQ$_k$. Examples of such spaces are Banach spaces with few operators and very few operators. A Banach space $\cZ$ has {\em few operators} if all operators in $\cB(\cZ)$ are of the form $\la I_\cZ +S$ with $\la\in\BC$ and $S$ strictly singular, and $\cZ$ has {\em very few operators} if all operators in $\cB(\cZ)$ are of the form $\la I_\cZ +K$ with $\la\in\BC$ and $K$ compact. All hereditary indecomposable Banach spaces have few operators \cite{GM93}, and example of a Banach space with very few operators was first constructed by Argyros and Haydon \cite{AH11}.
\end{remark}

\section{Analysis of Problem \ref{P:BSOP}}\label{S:AnalyseBSOP}

In this section we analyse Problem \ref{P:BSOP}. Throughout, let $\cV$, $\cW$, $\cZ_1$ and $\cZ_2$ be Banach spaces satisfying \eqref{BanProps}. We begin with some general results and a few corollaries, after which we consider some special cases in Subsections \ref{SubS:Z1Z2findim} and \ref{SubS:EssInc}.

\begin{lemma}\label{L:Z1Z2iso}
Assume $\cZ_1\simeq \cZ_2$. Then Problem \ref{P:BSOP} is solvable.
\end{lemma}

\begin{proof}[\bf Proof]
Taking $A_{22}\in\cB(\cZ_1,\cZ_2)$ invertible and setting $A_{12}$, $A_{21}$, $B_1$ and $B_2$ equal to zero-operators yields $T$ in \eqref{T} invertible.
\end{proof}

\begin{corollary}\label{C:VorWfindim}
Assume $\dim(\cV)<\infty$ or $\dim(\cW)<\infty$. Then $\cZ_1\simeq \cZ_2$ and hence Problem \ref{P:BSOP} is solvable.
\end{corollary}

\begin{proof}[\bf Proof]
Apply Lemma 3.4 of \cite{tHMRR19} to the first isomorphism in \eqref{BanProps} in case $\dim(\cV)<\infty$ and to the second isomorphism in \eqref{BanProps} in case $\dim(\cW)<\infty$.
\end{proof}

As indicated, by assumption invertible operators from $\cV\oplus\cZ_1$ to $\cV\oplus\cZ_2$ exist, the challenge is to find one with the left upper corner as in \eqref{T}. Hence we are interested in the  question which operators are contained in the set
\[
\fK:=\{I_\cV-B_1B_2\colon B_1\in\cB(\cW,\cV),\, B_2\in\cB(\cV,\cW)\}.
\]
Note that the set $\fK$ only depends on $\cV$ and $\cW$. Apart from zero operators, we can take $B_1$ and $B_2$ to be finite rank operators, in which case $I-B_1B_2$ is Fredholm with index 0.

\begin{lemma}\label{L:FredK}
Assume $\cV$ and $\cW$ are infinite dimensional. Then for any Fredholm operator $F$ in $\cB(\cV)$ with index 0 there exists an invertible operator $G\in\cB(\cV)$ so that $GF\in\fK$.
\end{lemma}

\begin{proof}[\bf Proof]
Since $\cV$ and $\cW$ are infinite dimensional all finite rank operators in $\cB(\cV)$ can be factored as $B_1B_2$ with $B_1\in\cB(\cW,\cV)$ and $B_2\in\cB(\cV,\cW)$. Let $F$ in $\cB(\cV)$ be Fredholm with index 0. By Theorem XI.5.3 in \cite{GGK90}, we have $F=H-K$ for operators $H,K\in\cB(\cV)$ with $H$ invertible and $K$ finite rank. Thus $H^{-1}F=I_\cV - H^{-1}K$. Since $H^{-1}K$ is also finite rank, it follows that $H^{-1}F\in\fK$. Hence we can take $G=H^{-1}$.
\end{proof}

It can happen that $\fK$ contains only Fredholm operators of index 0.

\begin{lemma}\label{L:fKonlyFredholm}
The set $\fK$ contains only Fredholm operators (of index 0) if and only if $\cV$ and $\cW$ are essentially incomparable Banach spaces.
\end{lemma}

\begin{proof}[\bf Proof]
This is directly clear from the definition of essentially incomparable.
\end{proof}

In Lemma \ref{L:fKonlyFredholm}, the parenthesised phrase can be included or removed without changing the validity of the statement.

In case $\fK=\cB(\cV)$, Problem \ref{P:BSOP} is solvable for any $\cZ_1$ and $\cZ_2$ so that \eqref{BanProps} hold. This case can easily be characterized.

\begin{lemma}\label{L:K=B(V)}
It holds that $\fK=\cB(\cV)$ if and only if $\cV$ is isomorphic to a complemented subspace of $\cW$. Furthermore, $\fK=\cB(\cV)$ holds if and only if $0\in\fK$.
\end{lemma}

\begin{proof}[\bf Proof]
Note that the condition $0\in\fK$ is equivalent to the existence of a left invertible operator from $\cV$ into $\cW$, whose range provides a complemented subspace of $\cW$ which is isomorphic to $\cV$. Conversely, if $\cV_0\subset \cW$ is isomorphic to $\cV$ and complemented in $\cW$, say with complement $\cW_0$ and isomorphism $R_0\in\cB(\cV,\cV_0)$, then $R=\sbm{R_0\\ 0}\in\cB(\cV,\cV_0\oplus\cW_0)=\cB(\cV,\cW)$ is left invertible, so that $0\in\fK$.

Now assume $R\in\cB(\cV,\cW)$ is left invertible, with left inverse $R^+\in\cB(\cW,\cV)$. Let $X\in\cB(\cV)$ be arbitrary. Then take $B_1=(I_\cV-X)R^+$ and $B_2=R$, and it follows that $X=I-B_1B_2\in\fK$. Hence $\fK=\cB(\cV)$. Conversely, it is clear that $\fK=\cB(\cV)$ implies that $0\in\cB(\cV)$.
\end{proof}

\begin{corollary}\label{C:IsoTocompSubs}
Assume $\cV$ is isomorphic to a complemented subspace of $\cW$. Then Problem \ref{P:BSOP} is solvable.
\end{corollary}

\begin{corollary}\label{C:V=W,Z1Z2findim}
Assume $\cV\oplus \cZ_1\simeq \cW\oplus \cZ_1$ or $\cV\oplus \cZ_2\simeq \cW\oplus \cZ_2$ (note that one implies the other via \eqref{BanProps}). Then Problem \eqref{P:BSOP} is solvable in case $\cZ_1$ or $\cZ_2$ is finite dimensional.
\end{corollary}

\begin{proof}[\bf Proof]
Without loss of generality we may assume $\dim (\cZ_1)<\infty$. Then Lemma 3.4 in \cite{tHMRR19} implies $\cV\simeq\cW$. By Lemma \ref{L:K=B(V)} we then find that $\fK=\cB(\cV)$, so that for $T$ we can take any invertible operator from $\cV\oplus \cZ_1$ to $\cV\oplus \cZ_2$.
\end{proof}

\subsection{The case where $\cZ_1$ and $\cZ_2$ are finite dimensional}\label{SubS:Z1Z2findim}

We start with a characterization of \eqref{BanProps} for the case that $\cZ_1$ and $\cZ_2$ are finite dimensional.

\begin{lemma}\label{L:BanPropChar-findim}
Assume $\cZ_1$ and $\cZ_2$ are finite dimensional. Then \eqref{BanProps} holds if and only if $\cV$ and $\cW$ are SUFDQ$_k$ for $k:=|\dim(\cZ_2)-\dim(\cZ_1)|$ or $k=0$.
\end{lemma}

\begin{proof}[\bf Proof] In case $k=0$ we have $\cZ_1\simeq \cZ_2$ and hence \eqref{BanProps} holds. Thus it remains to consider the case $k>0$.
We prove the claim for $\cV$ using the first isomorphism in \eqref{BanProps}; the claim for $\cW$ follows by an analogous argument. Assume \eqref{BanProps} holds and let $S\in\cB(\cV\oplus\cZ_1,\cV\oplus\cZ_2)$ be invertible. Decompose $S$ as
\begin{equation}\label{Sdec}
S=\mat{cc}{S_{11}&S_{12}\\ S_{21}&S_{22}}:\mat{c}{\cV\\ \cZ_1} \to \mat{c}{\cV\\ \cZ_2}.
\end{equation}
Since $S$ is invertible, $S$ is Fredholm with index 0. Write
\[
S=\wtil{S} + R \quad\mbox{with}\quad \wtilS=\mat{cc}{S_{11}&0\\ 0&0},\ R=\mat{cc}{0&S_{12}\\ S_{21}&S_{22}}.
\]
Since $R$ is finite rank, also $\wtil{S}$ is Fredholm with index 0. Therefore the dimension of $\kr S_{11}$ is equal to $\dim (\kr \wtil{S})-\dim (\cZ_1)$ and the dimension of $\cokr S_{11}$ is equal to $\dim (\cokr \wtil{S})-\dim (\cZ_2)$. From this we conclude that $S_{11}\in\cB(\cV)$ is Fredholm with index $\dim (\cZ_2) - \dim (\cZ_1)$. Hence $\cV$ is SUFDQ$_k$ by Proposition \ref{P:BanPropEquiv}.

Conversely, assume $\cV$ is SUFDQ$_k$. We have to consider two cases. First assume $\dim(\cZ_2)\geq \dim(\cZ_1)$, so that $k=\dim(\cZ_2)-\dim(\cZ_1)$. Let $F\in\cB(\cV)$ be a Fredholm operator of index $k$. Without loss of generality we may assume $F$ is surjective. To obtain an invertible operator $S$ as in \eqref{Sdec}, set $S_{11}=F$, $S_{12}=0$, let $S_{22}$ be any linear injection from $\cZ_1$ into $\cZ_2$. Then $\kr F$ and $\cokr S_{22}$ both have dimension $k$. Finally, take for $S_{21}$ an operator in $\cB(\cV,\cZ_2)$ with kernel equal to a complement of $\kr F$ while $S_{21}$ maps $\kr F$ isomorpically onto $\cokr S_{22}$. It is easy to see that $S$ constructed in this way is invertible. Hence $\cV\oplus\cZ_1\simeq\cV\oplus\cZ_2$. In case $\dim(\cZ_2)<\dim(\cZ_1)$, a similar construction based on an injective Fredholm operator with index $-k$ provides an invertible operator $S$ from $\cV\oplus\cZ_1$ to $\cV\oplus\cZ_2$.
\end{proof}

The above lemma implies in particular that $\cB(\cV)$ contains Fredholm operators of index $\pm k$. To solve Problem \ref{P:BSOP} we need to know if $\fK\subset \cB(\cV)$ contains Fredholm operators of index $k$ or $-k$.

\begin{proposition}\label{P:BSOPchar-findimZs}
Assume $\cZ_1$ and $\cZ_2$ are finite dimensional. Then Problem \ref{P:BSOP} is solvable if and only if $\fK$ contains a Fredholm operator of index $l:=\dim (\cZ_2) - \dim (\cZ_1)$.
\end{proposition}

In order to prove Proposition \ref{P:BSOPchar-findimZs} we first prove the following lemma, which shows that for $l\neq 0$ without loss of generality the Fredholm operator in Proposition \ref{P:BSOPchar-findimZs} is injective (if $l< 0$) or surjective (if $l> 0$).

\begin{lemma}\label{L:SurInjFredholm}
Assume $\fK$ contains a Fredholm operator of index $l\neq 0$. Then $\fK$ contains a Fredholm operator of index $l$ which is injective (if $l<0$) or surjective (if $l> 0$). Moreover, if $l< 0$, then $\cV$ and $\cW$ are SUFDQ$_k$ for $k:=|l|$.
\end{lemma}

\begin{proof}[\bf Proof]
Let $X=I_\cV-B_1B_2\in\fK$, with $B_1\in\cB(\cW,\cV)$ and $B_2\in\cB(\cV,\cW)$, be Fredholm with index $l$. Since $X\in\cB(\cV)$, we know $\cV$ is SUFDQ$_k$ for $k=|l|$. Since $X$ is Schur coupled to $Y:=I_\cW-B_2B_1 \in\cB(\cW)$, via the operator matrix $\sbm{I_\cV&B_1\\B_2&I_\cW}$, it follows \cite[p.\ 211]{BGKR05} that $Y$ is Fredholm with index $l$, hence also $\cW$ is SUFDQ$_k$. This proves the last claim of the proposition. Next we construct a Fredholm operator $\wtilX\in\fK$ of index $l$ which is injective or surjective.

There exists a finite rank operator $R\in\cB(\cV)$, say $R$ has rank $r$, so that $\wtilX:=X-R$ is injective (if $l\leq 0$) or surjective (if $l\geq 0$). In that case $\wtilX$ is also Fredholm with index $l$. Determine an integer $n>0$ so that $r\leq n k$ and note that $\cW$ is also SUFDQ$_{nk}$ and hence SUFDS$_{nk}$. Let $\cZ\subset\cV$ be a subspace of dimension $nk$ that contains the range of $R$. Then $\cW\simeq \cW\oplus \cZ$. Let $L\in\cB(\cW\oplus\cZ,\cW)$ be an isomorphism. Since $\cZ$ is complemented in $\cV$, the embedding $\tau_\cZ\in\cB(\cZ,\cV)$ of $\cZ$ into $\cV$, along an arbitrary complement of $\cZ$, has a left inverse $\pi_\cZ \in\cB(\cV,\cZ)$. Now set
\[
\wtilB_1=\mat{cc}{B_1&\tau_\cZ}L^{-1}\in\cB(\cW,\cV) \ands
\wtilB_2=L\mat{c}{B_2\\ \pi_\cZ R}\in\cB(\cV,\cW).
\]
We then obtain that
\[
\wtilX=I_\cV-B_1B_2-R=I_\cV-\wtilB_1\wtilB_2\in\fK.\qedhere
\]
\end{proof}

\begin{proof}[\bf Proof of Proposition \ref{P:BSOPchar-findimZs}]
In case $l=0$, note that $I_\cV\in\fK$, hence $\fK$ contains a Fredholm operator index $l$. On the other hand, for $l=0$ we have $\cZ_1\simeq\cZ_2$, so that Problem \ref{P:BSOP} is solvable by Lemma \ref{L:Z1Z2iso}. In the remainder of the proof assume $l\neq 0$.

Assume Problem \ref{P:BSOP} is solvable and $T$ as in \eqref{T} is invertible. Then the left upper corner of $T$ is in $\fK$ and reasoning as in the first paragraph of the proof of Lemma \ref{L:BanPropChar-findim} this operator is Fredholm with index $l$.

Conversely, assume $\fK$ contains a Fredholm operator $F$ with index $l$. According to Lemma \ref{L:SurInjFredholm}, we may assume $F$ is surjective if $l\geq 0$, or injective if $l\leq 0$. Assume $l\leq 0$, so that $F$ is injective and the range of $F$ has codimension $-l$. Let $\cV'$ be any complement of the range of $F$. Decompose $\cZ_1=\cZ_1'\oplus\cZ_1''$  with $\cZ_1'\simeq \cV'$ and $\cZ_1''\simeq\cZ_2$. Let $A'_{12}\in\cB(\cZ_1',\cV)$ be a left invertible operator with range $\cV'$ and $A'_{22}\in\cB(\cZ_1'',\cZ_2)$ an isomorphism. Then Problem \ref{P:BSOP} is solvable via the invertible operator
\[
T=\mat{ccc}{F& A'_{12}&0\\ 0&0& A'_{22}}:\mat{c}{\cV\\ \cZ_1' \\ \cZ_1''} \to \mat{c}{\cV\\ \cZ_2}.
\]
In case $l\geq 0$ a similar construction provides an invertible $T$ as in \eqref{T} with $A_{12}=0$.
\end{proof}

\begin{corollary}
Assume $\cV$ and $\cW$ are such that $\fK$ contains a Fredholm operator of index $l$. Then \eqref{BanProps} holds and Problem \ref{P:BSOP} is solvable for all finite dimensional Banach spaces $\cZ_1$ and $\cZ_2$ so that $\dim (\cZ_2) -\dim(\cZ_1)=n l$ for a positive integer $n$.
\end{corollary}

The next corollary follows directly by combining Proposition \ref{P:BSOPchar-findimZs} with Lemma \ref{L:fKonlyFredholm}.

\begin{corollary}\label{C:EssIncFinDimZs}
Assume $\cV$ and $\cW$ are essentially incomparable and $\cZ_1$ and $\cZ_2$ are finite dimensional. Then Problem \ref{P:BSOP} is solvable if and only if $\dim(\cZ_1)=\dim(\cZ_2)$.
\end{corollary}

We now employ Proposition \ref{P:BSOPchar-findimZs} to obtain a solution criterion that is independent of the set $\fK$.

\begin{proposition}\label{P:BSOPfindimiso}
Assume $\cZ_1$ and $\cZ_2$ are finite dimensional. Problem \ref{P:BSOP} is solvable in case $\cV$ has a finite codimensional subspace that is isomorphic to a complemented subspace of $\cW$ or in case $\cW$ has a finite codimensional subspace that is isomorphic to a complemented subspace of $\cV$.
\end{proposition}

\begin{proof}[\bf Proof]
In case $\cZ_1\simeq\cZ_2$ we are done. So assume $k:=|\dim (\cZ_2) - \dim (\cZ_1)|>0$. By Lemma \ref{L:BanPropChar-findim}, $\cV$ and $\cW$ are SUFDQ$_k$. Assume $\cV=\cV_1\oplus\cV_2$  and $\cW=\cW_1\oplus\cW_2$ with $\cV_1\simeq \cW_1$. If $\dim(\cV_2)<\infty$, then $\cV_1$ is SUFDQ$_k$ by applying Lemma \ref{L:BSPfindimcomp} to the first isomorphism in \eqref{BanProps}, and hence $\cW_1$ is SUFDQ$_k$. In case $\dim(\cW_2)<\infty$, we get that $\cV_1$ and $\cW_1$ are SUFDQ$_k$ by applying the same argument to the second isomorphism in \eqref{BanProps}. Under either of the conditions in the proposition we thus get that $\cV_1$ and $\cW_1$ are SUFDQ$_k$ and it suffices to prove that Problem \ref{P:BSOP} is solvable from this fact.

Now let $T_1\in\cB(\cV_1)$ be a Fredholm operator of index $l:=\dim (\cZ_2) - \dim (\cZ_1)$, which exists by Proposition \ref{P:BanPropEquiv}. Let $S_1\in\cB(\cV_1,\cW_1)$ be invertible. Now set
\begin{align*}
B_1&=\mat{cc}{S_1^{-1}&0\\ 0&0}:\mat{c}{\cW_1\\\cW_2}\to \mat{c}{\cV_1\\\cV_2},\\
B_2&=\mat{cc}{S_1 (T_1-I_{\cV_1})&0\\0&0}:\mat{c}{\cV_1\\\cV_2}\to \mat{c}{\cW_1\\\cW_2}.
\end{align*}
It then follows that
\[
I_{\cV}-B_1 B_2 =\mat{cc}{T_1&0\\ 0& I_{\cV_2}}\in\fK,
\]
and thus $I_{\cV}-B_1 B_2$ is Fredholm with index $l$. Hence Problem \ref{P:BSOP} is solvable by Proposition \ref{P:BSOPchar-findimZs}.
\end{proof}

In Proposition \ref{P:BSOPfindimiso} one cannot remove the condition that one of the subspaces of $\cV$ and $\cW$ has finite codimension, as shown in the next example.

\begin{example}\label{E:EAEnotSC}
Consider $\cV=\cV_0\oplus \ell^p$ and $\cW=\cV_0\oplus \ell^q$ with $1<p\neq q<\infty$ and where $\cV_0$ is an infinite dimensional Banach space with few operators, see Remark \ref{R:BSexamples}, in particular $\cV_0$ is SUFDQ$_k$ for no $k>0$. Since $\ell^p$ and $\ell^q$ are prime, they are SUFDQ, and thus $\cV$ and $\cW$ are SUFDQ, by Lemma \ref{L:BSPsum}. Hence for all finite dimensional $\cZ_1$ and $\cZ_2$ we have \eqref{BanProps}. Although $\cV$ and $\cW$ are not essentially incomparable, we claim that still all Fredholm operators in $\fK$ have index $0$.

We claim that $\cV_0$ and $\ell^p$ are projection incomparable. Indeed, assume this is not the case and $\cK$ is an infinite dimensional complemented subspace of $\ell^p$ which is isomorphic to a complemented subspace of $\cV_0$. Since $\ell^p$ is prime, $\ell^p$ and $\cK$ are isomorphic, so that $\cV_0$ contains a complemented copy of $\ell^p$. But then $\cV_0$ would be SUFDQ. The contradiction proves our claim. Likewise $\cV_0$ and $\ell^q$ are projection incomparable. In fact, by Theorem 7.97 in \cite{A04}, since all operators on $\cV_0$ are either Fredholm with index 0 or inessential (in fact strongly singular), $\cV_0$ and $\ell^p$ are essentially incomparable, and so are $\cV_0$ and $\ell^q$.

From the above it follows that $B_1\in \cB(\cV_0\oplus\ell^q,\cV_0\oplus\ell^p)$ and $B_2\in \cB(\cV_0\oplus\ell^p,\cV_0\oplus\ell^q)$ have the form
\[
\begin{aligned}
B_1&=\mat{cc}{\la_1 I_{\cV_0} -S_1 & B_{12}^{(1)}\\ B_{21}^{(1)} & B_{22}^{(1)} }:\mat{c}{\cV_0\\ \ell^q}\to \mat{c}{\cV_0\\ \ell^p},\\
B_2&=\mat{cc}{\la_2 I_{\cV_0} -S_2 & B_{12}^{(2)}\\ B_{21}^{(2)} & B_{22}^{(2)} }:\mat{c}{\cV_0\\ \ell^p}\to \mat{c}{\cV_0\\ \ell^q}
\end{aligned}
\]
with $\la_1,\la_2\in\BC$, $S_1$, $S_2$ strictly singular (hence inessential) and all $B_{ij}^{(1)}$ and $B_{ij}^{(2)}$ inessential. Since the class of inessential operators is closed under left and right multiplication with any bounded operator, it follows that
\[
I-B_1B_2=\mat{cc}{(1-\la_1\la_2)I_{\cV_0} &0 \\ 0& I_{\ell^p}} -B
\]
where $B$ is an inessential operator in $\cB(\cV_0\oplus\ell^p)$. In particular, $I-B_1B_2$ is Fredholm with index 0 in case $\la_1\la_2\neq 0$ and not Fredholm in case $\la_1\la_2=0$ (for otherwise $\sbm{0&0\\0&I_{\ell^p}}$ would be Fredholm). Hence, in this case also all Fredholm operators in $\fK$ have index 0, so that Problem \ref{P:BSOP} with $\cV$ and $\cW$ as in this example and $\cZ_1$ and $\cZ_2$ finite dimensional is only solvable in case $\cZ_1\simeq \cZ_2$.
\end{example}

\subsection{Essentially incomparable Banach spaces}\label{SubS:EssInc}

In Lemma \ref{L:fKonlyFredholm} and Corollary \ref{C:EssIncFinDimZs} above we already encountered the notion of essentially incomparable Banach spaces. Next we consider some other cases of Problem \ref{P:BSOP} where essential incomparability plays a role. All the results rely on the following general lemma.


\begin{lemma}\label{L:EssIncImplic}
Let $S_1,\cS_2,\cT_1,\cT_2$ be Banach spaces so that $\cS_1\oplus\cT_1 \simeq \cS_2\oplus\cT_2$ and so that $\cT_1$ and $\cT_2$ are essentially incomparable. Then there is a subspace of $\cT_1$ with finite codimension which is isomorphic to a complemented subspace of $\cS_2$. In particular, if $\cT_1$ or $\cS_2$ is SUFDQ, then $\cS_2$ contains a complemented copy of $\cT_1$.
\end{lemma}

\begin{proof}[\bf Proof]
Let $T$ be an isomorphism from $\cS_1\oplus\cT_1$ to $\cS_2\oplus\cT_2$ and decompose $T$ and $T^{-1}$ as
\[
T=\mat{cc}{T_{11}&T_{12}\\T_{21}&T_{22}}\mat{c}{\cS_1\\ \cT_1} \to \mat{c}{\cS_2\\ \cT_2},\
T^{-1}=\mat{cc}{S_{11}&S_{12}\\S_{21}&S_{22}}\mat{c}{\cS_2\\ \cT_2} \to \mat{c}{\cS_1\\ \cT_1}.
\]
Since $\cT_1$ and $\cT_2$ are essentially incomparable, $T_{22}$ and $S_{22}$ are inessential operators. Writing out $T^{-1}T=I$ in terms of the decompositions of $T$ and $T^{-1}$ we find that $I_{\cT_1}=S_{21}T_{12}+S_{22}T_{22}$. Hence $S_{21}T_{12}=I_{\cT_1}-S_{22}T_{22}$, which implies $S_{21}T_{12}$ is Fredholm with index 0, because $T_{22}$ and $S_{22}$ are inessential. This in turn implies that $T_{12}$ is left Atkinson, that is, $T_{12}$ has finite dimensional kernel and complemented range, cf., Definition 7.1 and Theorem 7.2 in \cite{A04} (noting that we can replace compact operators in Definition 7.1 by inessential operators). Likewise, $S_{21}$ is right Atkinson, but we do not need this fact in the current proof. Let $\cT_1'\subset\cT_1$ be any complement of $\kr T_{12}$, so that $\cT_1'$ has finite codimension.  Then $T_{12}$ restricted to $\cT_1'$ is an isomorphism from $\cT_1'$ onto the range of $T_{12}$, which is a complemented subspace of $\cS_2$. This proves the first claim.

If $\cT_1$ is SUFDQ, then $\cT_1$ is isomorphic to $\cT_1'$ and hence to a complemented subspace of $\cS_2$. Assume $\cS_2$ is SUFDQ. Note that
\[
\cT_1\simeq \cT_1'\oplus \kr T_{12}\simeq T_{12}(\cT_1')\oplus \kr T_{12}.
\]
Let $S_2'$ be a complement of $T_{12}(\cT_1')$ in $\cS_2$, so that $\cS_2=T_{12}(\cT_1')\oplus\cS_2'$. Since $\dim(\kr T_{12})< \infty$, we have
\[
\cS_2\simeq \cS_2\oplus\kr T_{12}
=T_{12}(\cT_1')\oplus\cS_2' \oplus\kr T_{12}
\simeq T_{12}(\cT_1')\oplus\kr T_{12}\oplus\cS_2'
\simeq \cT_1\oplus\cS_2'.
\]
Hence also in this case $\cS_2$  contains a complemented copy of $\cT_1$.
\end{proof}

In Corollary \ref{C:EssIncFinDimZs} we concluded that in case $\cV$ and $\cW$ are essentially incomparable and $\cZ_1$ and $\cZ_2$ are finite dimensional, them Problem \ref{P:BSOP} is only solvable in case $\dim(\cZ_1)=\dim(\cZ_2)$. We now consider another case with $\cV$ and $\cW$ essentially incomparable, where $\cZ_1$ and $\cZ_2$ are infinite dimensional.

\begin{lemma}\label{L:VW-EssInc}
Assume $\cV\oplus \cZ_1\simeq \cW\oplus \cZ_1$ or $\cV\oplus \cZ_2\simeq \cW\oplus \cZ_2$ (note that one implies the other via \eqref{BanProps}). Then $\cZ_1\simeq\cZ_2$ and Problem \ref{P:BSOP} is solvable if $\cV$ and $\cW$ are essentially incomparable and one of the following conditions is satisfied
\begin{itemize}
  \item[(i)] $\cV\simeq \cV\oplus\cV$ and either $\cV$ is SUFDQ or both $\cZ_1$ and $\cZ_2$ are SUFDQ;
  \item[(ii)] $\cW\simeq \cW\oplus\cW$ and either $\cW$ is SUFDQ or both $\cZ_1$ and $\cZ_2$ are SUFDQ.
\end{itemize}
\end{lemma}

\begin{proof}[\bf Proof]
Assume $\cV$ and $\cW$ are essentially incomparable and condition (i) holds; the proof for (ii) goes analogously. For $i=1,2$, apply Lemma \ref{L:EssIncImplic} with $\cS_1=\cS_2=\cZ_i$, $\cT_1=\cV$ and $\cT_2=\cW$. We find that $\cZ_i\simeq\cV\oplus\cZ_i'$ for some Banach space $\cZ_i'$. Since also $\cV\simeq \cV\oplus\cV$, we have
\[
\cZ_i\simeq \cV\oplus\cZ_i' \simeq \cV\oplus\cV\oplus\cZ_i' \simeq \cV\oplus\cZ_i.
\]
Using the first isomorphism in \eqref{BanProps} we obtain that $\cZ_1\simeq \cV \oplus\cZ_1\simeq \cV\oplus\cZ_2 \simeq \cZ_2$, which proves our claim.
\end{proof}

Next we consider the case where $\cZ_1$ and $\cZ_2$ are essentially incomparable.

\begin{lemma}\label{L:Z1Z2EssInc}
Assume $\cZ_1$ and $\cZ_2$ are essentially incomparable and  $\cZ_1$ and $\cZ_2$ are SUFDQ. Then Problem \ref{P:BSOP} is solvable in case $\cZ_1\simeq \cZ_1\oplus \cZ_1$ and $\cZ_2\simeq \cZ_2\oplus \cZ_2$.
\end{lemma}

\begin{proof}[\bf Proof]
Applying Lemma \ref{L:EssIncImplic} to \eqref{BanProps} it follows that $\cV$ contains complemented copies of $\cZ_1$ and $\cZ_2$ and likewise $\cW$ contains complemented copies of $\cZ_1$ and $\cZ_2$. Hence for $i=1,2$ we have $\cV\simeq\cV_i\oplus\cZ_i$ and $\cW\simeq\cW_i\oplus \cZ_i$ for Banach spaces $\cV_i$ and $\cW_i$. In particular, we find that
\[
\cV_1\oplus\cZ_1\simeq \cV_2\oplus\cZ_2.
\]
Again applying Lemma \ref{L:EssIncImplic} it follows that $\cV_1$ contains a complemented copy of $\cZ_2$ and $\cV_2$ contains a complemented copy of $\cZ_2$. In both cases we obtain that $\cV\simeq\wtil{\cV}\oplus\cZ_1\oplus\cZ_2$ for a Banach space $\wtil{\cV}$ and by similar arguments
$\cW\simeq\wtil{\cW}\oplus\cZ_1\oplus\cZ_2$ for a Banach space $\wtil{\cW}$. By assumption we have $\cZ_1\simeq \cZ_1\oplus \cZ_1$ and $\cZ_2\simeq \cZ_2\oplus \cZ_2$. Thus
\[
\cZ_1\oplus\cZ_2\oplus \cZ_1 \simeq \cZ_1\oplus \cZ_1\oplus\cZ_2 \simeq \cZ_1\oplus\cZ_2
\simeq \cZ_1\oplus\cZ_2 \oplus\cZ_2.
\]
Now let
\[
S=\mat{cc}{S_{11}& S_{12}\\ S_{21} & S_{22}}:\mat{c}{\cZ_1\oplus\cZ_2\\ \cZ_1}\to \mat{c}{\cZ_1\oplus\cZ_2\\ \cZ_2}
\]
be invertible. Subject to the identifications $\cV\simeq\wtil{\cV}\oplus\cZ_1\oplus\cZ_2$ and
$\cW\simeq\wtil{\cW}\oplus\cZ_1\oplus\cZ_2$ we obtain an invertible operator
\[
T=\mat{ccc}{I_{\wtil{\cV}}&0&0\\ 0&S_{11}& S_{12}\\ 0&S_{21}&S_{22}}
:\mat{c}{\wtil{\cV}\\ \cZ_1\oplus\cZ_2\\\cZ_1}\to\mat{c}{\wtil{\cV}\\ \cZ_1\oplus\cZ_2\\\cZ_2}
\]
which is as in \eqref{T} via
\begin{align*}
&B_1=\mat{cc}{0&0\\0& I_{\cZ_1\oplus\cZ_2}}:\mat{c}{\wtil{\cW}\\ \cZ_1\oplus\cZ_2} \to \mat{c}{\wtil{\cV}\\ \cZ_1\oplus\cZ_2},\\
&B_2=\mat{cc}{0&0\\0& I_{\cZ_1\oplus\cZ_2}-S_{11}}:\mat{c}{\wtil{\cV}\\ \cZ_1\oplus\cZ_2} \to \mat{c}{\wtil{\cW}\\ \cZ_1\oplus\cZ_2},\\
& A_{12}=\mat{c}{0\\S_{12}}:\cZ_1\to\mat{c}{\wtil{\cV}\\\cZ_1\oplus\cZ_2},\
A_{21}=\mat{cc}{0&S_{21}}:\mat{c}{\wtil{\cV}\\\cZ_1\oplus\cZ_2}\to\cZ_2,\\
&A_{22}=S_{22}.\qedhere
\end{align*}
\end{proof}

\begin{remark}
In Lemmas \ref{L:VW-EssInc} and \ref{L:Z1Z2EssInc} we encountered Banach spaces $\cZ$ with the property $\cZ\simeq\cZ\oplus \cZ$. This property also appears in the Pe{\l}czy\'{n}ski decomposition technique for the Banach space Schroeder-Bernstein problem, cf., Theorem 2.2.3 in \cite{AK16}. All separable Banach spaces with countable primary unconditional bases have this property \cite{K99}, see also \cite{FR05} for further examples. Prime spaces that have infinite dimensional complemented subspaces with infinite codimension are isomorphic to their squares. However, there are prime spaces whose infinite dimensional complemented subspaces are all finite codimensional \cite{M03}, so these prime spaces cannot be isomorphic to their squares. Other examples are a reflexive Banach space not isomorphic to its square \cite{F72} and a Banach space that is isomorphic to its cube, but not to its square \cite{G96}.
\end{remark}

\section{New cases where EAE and SC coincide, or not}\label{S:EAEvsSC}

In this section we translate some of the results from the analysis of Problem \ref{P:BSOP} to the question in which cases EAE implies SC for relatively regular operators. Let $U\in\cB(\cX)$ and $V\in\cB(\cY)$ be relatively regular operators as in \eqref{UVdec}. For the reader's convenience we recall the setting and translation to Problem \ref{P:BSOP} here. Hence $U$ and $V$ decompose as
\[
U=\mat{cc}{U'&0\\ 0&0}:\mat{c}{\cX_1\\ \cX_2}\to \mat{c}{\cX_1'\\ \cX_2'} \ands
V=\mat{cc}{V'&0\\ 0&0}:\mat{c}{\cY_1\\ \cY_2}\to \mat{c}{\cY_1'\\ \cY_2'},\\
\]
where $U'$ and $V'$ are invertible and
\[
\cX_2=\kr U,\quad \cX_1'=\ran U,\quad
\cY_2=\kr V,\quad \cY_1'=\ran V.
\]
Although not uniquely determined, we will use the notation
\[
\cX_2'=\cokr U \ands \cY_2'=\cokr V.
\]
Since $U$ and $V$ are relatively regular, we have
\[
\cX_1\simeq \cX_1' \ands \cY_1\simeq \cY_1'.
\]
Throughout this section we will assume $U$ and $V$ are EAE, which yields
\[
\cX_2\simeq\cY_2 \ands \cX_2'\simeq\cY_2'.
\]
Hence in that case, up to isomorphy we only have the spaces
\[
\cV:=\ran U,\quad
\cW:=\ran V,\quad
\cZ_1:=\kr U\simeq\kr V,\quad
\cZ_2:=\cokr U \simeq\cokr V,
\]
and this also provides the translation to Problem \ref{P:BSOP}, noting that we may interchange the roles of $U$ and $V$, and hence $\cV$ and $\cW$.

Using the above we translate some results of Section \ref{S:AnalyseBSOP} to the current setting.

\begin{proposition}\label{P:Implic}
Let $U\in\cB(\cX)$ and $V\in\cB(\cY)$ be relatively regular operators as in \eqref{UVdec} which are EAE. Then $U$ and $V$ are SC in the following cases:
\begin{itemize}
\item[(i)] $\kr U\simeq \cokr U$ (Lemma \ref{L:Z1Z2iso});

\item[(ii)] $U$ or $V$ finite rank (Corollary \ref{C:VorWfindim});

\item[(iii)] $\ran U$ isomorphic to a complemented subspace of $\ran V$ or $\ran V$ isomorphic to a complemented subspace of $\ran U$ (Corollary \ref{C:IsoTocompSubs});

\item[(iv)] $\cX\simeq\cY$, $U$ and $V$ Atkinson operators (Corollary \ref{C:V=W,Z1Z2findim});

\item[(v)] $\kr U$ and $\cokr U$ finite dimensional, $\ran U$ has a finite codimensional subspace which is isomorphic to a complemented subspace of $\ran V$ or $\ran V$ has a finite codimensional subspace which is isomorphic to a complemented subspace of $\ran U$ (Proposition \ref{P:BSOPfindimiso});

\item[(vi)] $\ran U$ and $\ran V$ essentially incomparable, $\kr U$ and $\cokr U$ finite dimensional with the same dimension (Corollary \ref{C:EssIncFinDimZs});

\item[(vii)] $\ran U$ and $\ran V$ essentially incomparable, $\ran U\simeq \ran U\oplus\ran U$ and either $\ran U$ is SUFDQ or both $\kr U$ and $\cokr U$ are SUFDQ or $\ran V\simeq \ran V\oplus\ran V$ and either $\ran V$ is SUFDQ or both $\kr V$ and $\cokr V$ are SUFDQ (Lemma \ref{L:VW-EssInc});

\item[(viii)] $\kr U$ and $\cokr U$ essentially incomparable, $\kr U$ and $\cokr U$ SUFDQ, $\kr U\simeq \kr U \oplus \kr U$ and $\cokr U\simeq \cokr U \oplus \cokr U$ (Lemma \ref{L:Z1Z2EssInc}).

\end{itemize}
\end{proposition}

We point out here that case (i) was obtained in \cite[Theorem 3.3]{tHR13}, while case (iv) can easily be derived from the main result in \cite{tHMRR19} (noting that with $\kr U$ and $\cokr U$ finite dimensional we have that $\cX$ and $\cY$ are essentially incomparable if and only if $\ran U$ and $\ran V$ are essentially incomparable). Case (iv) is an improvement of Proposition 5.1 (1) where in addition to $U$ and $V$ Atkinson it was also asked that $\cX$ and $\cY$ are SUFDQ. All other cases discussed in Proposition \ref{P:Implic} do not seem to be covered in the literature, to the best of our knowledge.

Corollary \ref{C:EssIncFinDimZs} also provides a way to construct an example that shows EAE does not imply SC, with $\cX$ and $\cY$ essentially incomparable and $\kr U$ and $\cokr U$ finite dimensional, but this, and much more, was already covered in \cite{tHMRR19}. The main result of \cite{tHMRR19}, as well as \cite[Proposition 2.4]{tHMRR19}, may lead one to hope that for relatively regular operators the case with $\cX$ and $\cY$ essentially incomparable is the only case where EAE and SC do not coincide. Although maybe somewhat contrived, the following example shows this is not the case.

\begin{example}
Following the construction preceding Corollary \ref{C:EAEnotSC} we can use Example \ref{E:EAEnotSC} to obtain a new example which shows EAE need not imply SC. In Example \ref{E:EAEnotSC} we take
\[
\cV=\cV_0\oplus \ell^p \ands \cW=\cV_0\oplus \ell^q
\]
with $\cV_0$ a Banach space with few operators and $p\neq q$. Since $\ell^q$ and $\ell^p$ are SUFDQ, so are $\cV$ and $\cW$, by Lemma \ref{L:BSPsum}, hence $\cV\oplus\cZ\sim \cV$ and $\cW\oplus\cZ\sim \cW$ for any finite dimensional Banach space $\cZ$, hence we may take $\cX=\cW$ and $\cY=\cV$. We then set
\[
U=\mat{cc}{I_{\cV_0}&0\\0&S_q} \ands V=\mat{cc}{I_{\cV_0}&0\\0&S_p},
\]
with $S_q$ and $S_p$ the forward shifts on $\ell^q$ and $\ell^p$, respectively. Translating back to Problem \ref{P:BSOP} we obtain $\cZ_1 =\{0\}$ and $\cZ_2\simeq \BC$, so that Problem \ref{P:BSOP} is not solvable by the analysis in Example \ref{E:EAEnotSC}, hence $U$ and $V$ are not SC. On the other hand, $U$ and $V$ are Fredholm and the dimensions of their kernels and cokernels coincide, so that $U$ and $V$ are EAE by Proposition 1 in \cite{BT92a}.
\end{example}

\paragraph{\bf Declarations:} $ $\smallskip

\paragraph{\bf Conflict of interest} The authors declare that they have no conflict of interest.\smallskip

\paragraph{\bf Funding}
This work is based on research supported in part by the National Research Foundation of South Africa (NRF) and the DSI-NRF Centre of Excellence in Mathematical and Statistical Sciences (CoE-MaSS). Any opinion, finding and conclusion or recommendation expressed in this material is that of the authors and the NRF and CoE-MaSS do not accept any liability in this regard.\smallskip

\paragraph{\bf Availability of data and material} Data sharing is not applicable to this article as no datasets were generated or analysed during the current study.\smallskip

\paragraph{\bf Code availability} Not applicable.


\end{document}